\ifx \processToSlides \undefined
  \ifx\usepackage\RequirePackage	% documentclass already done...  Assume
    \let\usingAmsArtXII\usepackage	% this file is processed by amsart
  \fi
\else
  \documentclass[12pt]{amsart}
  \usepackage{amssymb,amscd}
  \usepackage[paperwidth=11truein,paperheight=8.25truein,margin=1mm,includeheadfoot]{geometry}
  \def \useHugeSize {}
  \def \numberingIsThrough {}
\fi

\ifx \labelsONmargin\undefined

  \ifx\RequirePackage\undefined
    \expandafter\ifx\csname amsart.sty\endcsname\relax
        \documentstyle[12pt,amssymb,amscd]{amsart}
    \fi

    \def\mathbb{\Bbb}
    \def\mathfrak{\frak}
    \def\mathbf{\bold}
    \ifx\boldsymbol\undefined	% I *think* some very old AMS* did not have it
      \def\boldsymbol#1{{\bold #1}}
    \fi

    \def\mathbit{\boldsymbol}

    \makeatletter
    \newenvironment{proof}{%
         \@ifnextchar[{%
                       \expandafter\let\expandafter\end@proof
                         \csname endpf*\endcsname
                         \my@proof
                      }{\let\end@proof\endpf\pf}%
        }{\end@proof}
    \def\my@proof[#1]{\@nameuse{pf*}{#1}}
    \makeatother

    \def\xrightarrow[#1]#2{@>{#2}>{#1}>}
    \def\xleftarrow[#1]#2{@<{#2}<{#1}<}

    \def\providecommand#1{\def#1}
    \def\emph#1{{\em #1}}
    \def\textbf#1{{\bf #1}}

    \def\mathring{\overset{\,\,{}_\circ}}% For slanted letters only, sub too high
  \else
      \ifx\usepackage\RequirePackage
      \else	%% <<<<<<<<<<<<<<<<<<<<<<<<<<<<<<<< THIS IS THE MAIN BRANCH
        \documentclass[12pt]{amsart}		%% {{{{{{{ CHOSE ONE OF TWO:
        \usepackage{amssymb,amscd}
	\let\usingAmsArtXII\usepackage

      \fi
      \ifx \mathring\undefined		% Not AmS-LaTeX 2
        \DeclareMathAccent{\mathring}{\mathalpha}{operators}{"17}
      \fi

      \long\def\FAKEendPROOF{\endtrivlist}
      \ifx\endproof\FAKEendPROOF
	  \def\endproof{\qed\endtrivlist}
      \fi

      \ifx\mathbit\undefined	% They can collect their senses and define it
        \DeclareMathAlphabet{\mathbit}{OML}{cmm}{b}{it}
      \fi

      \ifx \usingAmsArtXII \undefined
      \else				% We loaded this package ourselves
	\advance\oddsidemargin -0.1\textwidth
	\evensidemargin\oddsidemargin
      \fi

      \def\Sb#1\endSb{_{\substack{#1}}}
      \def\Sp#1\endSp{^{\substack{#1}}}

  \fi
\makeatletter
\@ifundefined{DeclareFontShape} 
     {\@ifundefined{selectfont}
             {%    Here we are in non-NFSS
                \message{We need AmS-LaTeX, this version of LaTeX does not 
                        support it}\@@end
             }{%  Here we are in  NFSS1
                \def\mathcal{\cal}
                
                \def\pcyr{%
                        \def\default@family{UWCyr}%
                        \let\oldSl@\sl
                        \def\sl{\def\default@shape{it}\oldSl@}%
                        \cyracc
                        \language\Russian\family{UWCyr}\selectfont
                }
             }}{% Here we are in  NFSS2
                \DeclareFontEncoding{OT2}{}{\noaccents@}
                \DeclareFontSubstitution{OT2}{cmr}{m}{n}
                \DeclareFontFamily{OT2}{cmr}{\hyphenchar\font45 }
                \DeclareFontShape{OT2}{cmr}{m}{n}{%
                     <5><6><7><8><9><10>gen*wncyr %
                     <10.95><12><14.4><17.28><20.74><24.88> wncyr10 %
                }{}
                \DeclareFontShape{OT2}{cmr}{m}{it}{%
                     <5><6><7><8><9><10> gen * wncyi%
                     <10.95><12><14.4><17.28><20.74><24.88> wncyi10%
                }{}
                \DeclareFontShape{OT2}{cmr}{bx}{n}{%
                     <5><6><7><8><9><10> gen * wncyb%
                     <10.95><12><14.4><17.28><20.74><24.88> wncyb10%
                }{}
                \DeclareFontShape{OT2}{cmr}{m}{sl}{%
                     <-> ssub * cmr/m/it%
                }{}
                \DeclareFontShape{OT2}{cmr}{m}{sc}{%
                     <5><6><7><8><9><10>%
                     <10.95><12><14.4><17.28><20.74><24.88> wncysc10%
                }{}
                \DeclareFontFamily{OT2}{cmss}{\hyphenchar\font45 }
                \DeclareFontShape{OT2}{cmss}{m}{n}{%
                     <8><9><10> gen * wncyss%
                     <10.95><12><14.4><17.28><20.74><24.88> wncyss10%
                }{}
                
                \def\cyrencodingdefault{OT2}
                \def\pcyr{%
                        \cyracc
                        \let\encodingdefault\cyrencodingdefault
                        \language\Russian\fontencoding{OT2}\selectfont
                }
             }   

\@ifundefined{theorembodyfont} 
     {
        \def\theorembodyfont#1{\relax}
        \makeatletter
          \let\@@th@plain\th@plain
          \def\th@plain{ \@@th@plain \slshape }
        \makeatother
        \let\normalshape\relax
     }{}   

\ifx\cprime\undefined
     \def\cprime{$'$}
\fi
\makeatletter
  \def\@sect@my#1#2#3#4#5#6[#7]#8{%
\ifnum #2>\c@secnumdepth
   \let\@svsec\@empty
 \else
   \refstepcounter{#1}%
\edef\@svsec{\ifnum#2<\@m
             \@ifundefined{#1name}{}{\csname #1name\endcsname\ }\fi
\noexpand\rom{\csname the#1\endcsname.}\enspace}\fi
 \@tempskipa #5\relax
 \ifdim \@tempskipa>\z@ % then this is not a run-in section heading
   \begingroup #6\relax
   \@hangfrom{\hskip #3\relax\@svsec}{\interlinepenalty\@M #8\par}%
   \endgroup
   \if@article\else\csname #1mark\endcsname{%
        \ifnum \c@secnumdepth >#2\relax\csname the#1\endcsname. \fi#7}\fi
\ifnum#2>\@m \else
       \let\@tempf\\ \def\\{\protect\\}\addcontentsline{toc}{#1}%
{\ifnum #2>\c@secnumdepth \else
             \protect\numberline{%
               \ifnum#2<\@m
               \@ifundefined{#1name}{}{\csname #1name\endcsname\ }\fi
               \csname the#1\endcsname.}\fi
           #8}\let\\\@tempf
     \fi
 \else
  \def\@svsechd{#6\hskip #3\@svsec
    \@ifnotempty{#8}{\ignorespaces#8\unskip
       \ifnum\spacefactor<1001.\fi}%
        \ifnum#2>\@m \else
          \let\@tempf\\ \def\\{\protect\\}\addcontentsline{toc}{#1}%
            {\ifnum #2>\c@secnumdepth \else
              \protect\numberline{%
                \ifnum#2<\@m
                \@ifundefined{#1name}{}{\csname #1name\endcsname\ }\fi
                \csname the#1\endcsname.}\fi
             #8}\let\\\@tempf\fi}%
 \fi
\@xsect{#5}}
  \ifx\RequirePackage\undefined
  \let\@sect\@sect@my             % Cannot just comment the above
  \fi
  \ifx\RequirePackage\undefined	%              Not under LaTeX2e
  \def\th@remark@my{\theorempreskipamount6\p@\@plus6\p@
    \theorempostskipamount\theorempreskipamount
    \def\theorem@headerfont{\it}\normalshape}
    \let\th@remark\th@remark@my
  \else				%              Under LaTeX2e
    \let\o@@remark\th@remark

    \ifx\theorempostskipamount\undefined		% AmS-art
    \else						% Transformation groups
      \def\th@remark{\o@@remark
	\ifdim\theorempostskipamount < 2pt\relax
	  \theorempostskipamount\theorempreskipamount
	     \multiply\theorempostskipamount\tw@
	     \divide\theorempostskipamount\thr@@
	\fi
      }
    \fi
  \fi
\let\myLabel\@gobble
\def\labelsONmargin{\@mparswitchfalse\def\myLabel##1{\@bsphack\marginpar
                                  {\normalshape\tiny\rm Label ##1}\@esphack}}
\ifx \url\undefined
  \def\url#1{{\tt #1}}%
\fi
\def\PREpmodSKIP{\allowbreak  \if@display\mkern18mu\else\mkern8mu\fi}

\def\cyracc{\def\u##1{%\if \i##1\accent"24 i%
                \if \i##1\char"1A%
                \else \if I##1\char"12%
                \else \accent"24 ##1\fi\fi }%
\def\"##1{\if e##1{\char"1B}%
                \else \if E##1{\char"13}%
                \else \accent"7F ##1\fi\fi }%
\def\9##1{\if##1z\char"19 
\else\if##1Z\char"11 
\else\if##1E\char"03 
\else\if##1e\char"0B 
\else\if##1u\char"18 
\else\if##1U\char"10 
\else\if##1A\char"17 
\else\if##1a\char"1F 
\else\if##1p\char"7E 
\else\if##1P\char"5E 
\else\if##1Q\char"5F 
\else\if##1q\char"7F 
\else\if##1i\char"1A 
\else\if##1I\char"12 
\else\if##1N\char"7D 
\fi
\fi
\fi
\fi
\fi
\fi
\fi
\fi
\fi
\fi
\fi
\fi
\fi
\fi
\fi
}%
\def\cydot{{\kern0pt}}}%
\def\cydot{$\cdot$}
\ifx\Russian\undefined
        \def\Russian{0\relax
    \message{Don't know the hyphenation rules for Russian^^J
                        Please do INITeX with `input  russhyph' in the 
                        command line}%
                \gdef\Russian{0\relax}%
        }
\fi
\ifx\RequirePackage\undefined			% for 209
  \def\@putname#1#2#3#4{\def\@@ref{#3}\let\old@bf\bf
        \def\bf##1{\old@bf\if?\noexpand##1?{#4}\else##1\fi}%
	#1{#2}%
        \let\bf\old@bf}
\else						% for 2e
  \def\@putname#1#2#3#4{\def\@@ref{#3}\let\old@bf\bf	% for 209???
	\let\old@reset@font\reset@font			% for 2e
        \def\bf##1{\old@bf\if?\noexpand##1?{#4}\else##1\fi}%
	\def\reset@font##1##2{\old@reset@font##1\if?\noexpand##2?{#4}\else##2\fi}#1{#2}%
        \let\bf\old@bf\let\reset@font\old@reset@font}
\fi
\let\my@ref=\ref
\def\ref#1{\@putname\my@ref{#1}{#1}{\tiny\rm\@@ref}}
\let\my@pageref=\pageref
\def\pageref#1{\@putname\my@pageref{#1}{#1}{\tiny\rm\@@ref}}
\let\my@cite=\cite
\def\cite#1{\@putname\my@cite{#1}{\@citeb}{\tiny\rm\@@ref}}
\makeatother
\fi
\relax 

\ifx \SKIPstatementDEFS \undefined
  \theoremstyle{plain} % for references in unnumbered theorems
  \theorembodyfont{\sl}
\fi

\ifx\useDoubleSpacing\undefined
\else
  \usepackage{setspace}
\fi

\ifx \address \undefined
  \if \institute \undefined \else	% Springer Verlag multiauthor books
     \def\address{\institute}
     \ifx \email \undefined
        \let\email\texttt
     \fi
  \fi
\fi

\let\emphOrig\emph
\ifx \doEmphToIndex \undefined
  
\else
  \usepackage{makeidx}
  \makeindex

  \def\eatToBar#1|{}
  \def\emphToIndexSLASH#1\/{\index{#1}\eatToBar}
  \def\emphToIndexDOTSLASH#1.\/{\emphToIndexSLASH #1\/}
  \def\emphAndIndex#1{\emphOrig{#1}{\emphToIndexDOTSLASH #1.\/|}}
  \let\emph\emphAndIndex
\fi

\usepackage{xcolor}

\ifx \SKIPstatementDEFS \undefined

\ifx \numberingIsThrough \undefined
\numberwithin{equation}{section}
\fi

\theorembodyfont{\rm}
\theoremstyle{definition}

\ifx \numberingIsThrough \undefined
\newtheorem{definition}{Definition}[section]
\else
\newtheorem{definition}{Definition}
\fi

\newtheorem{example}[definition]{Example}

\theoremstyle{remark}

\newtheorem{remark}[definition]{Remark} %\renewcommand{\theremark}{}

\ifx \numberingIsThrough \undefined
\newtheorem{note}{Note}[section] 
\newtheorem{summary}{Summary}[section] 
\else

\fi

\theoremstyle{plain} % for future references
\theorembodyfont{\sl}

\newtheorem{theorem}[definition]{Theorem}
\newtheorem{lemma}[definition]{Lemma}
\newtheorem{corollary}[definition]{Corollary}
\newtheorem{proposition}[definition]{Proposition}

\fi %%% end: if \SKIPstatementDEFS \undefined

\def\supp{\operatorname {supp}}
\def\Hom{\operatorname{Hom}}

\def\ad{\operatorname{ad}}

\def\b{\beta}

\def\C{{\mathbb C}}

\def\g{{\mathfrak {g}}}
\def\b{{\mathfrak {b}}}
\def\q{{\mathfrak {q}}}

\def\k{{\mathfrak {k}}}
\def\l{{\mathfrak {l}}}
\def\m{{\mathfrak {m}}}
\def\n{{\mathfrak {n}}}

\def\p{{\mathfrak {p}}}

\def\h{{\mathfrak {h}}}

\ifx\useHugeSize\undefined
\else
\Huge
\fi

\relax

\begin{document}
\title[Simple weight modules over polynomial vector fields]{Simple weight modules with finite weight multiplicities over the Lie algebra of polynomial vector fields}

\author{Dimitar Grantcharov} 
\address{Department of Mathematics \\ University of Texas at Arlington \\ Arlington, TX 76019} \email{grandim@uta.edu}
\thanks{The first author is partially supported by Simons Collaboration Grant 358245}
\author{Vera Serganova} 
\address{Department of Mathematics \\ University of California at Berkeley \\ Berkeley, CA 94720} \email{serganov@math.berkeley.edu}
\thanks{The second author is partially supported by NSF Grant 1701532}
\date{ \today }

\maketitle
\begin{abstract} Let ${\mathcal W}_n$ be the Lie algebra of polynomial vector fields.
  We classify simple weight ${\mathcal W}_n$-modules $M$ with finite weight multiplicities. We prove that every such nontrivial module $M$ is either a tensor module or the unique simple submodule in a tensor module associated
  with the de Rham complex on $\mathbb C^n$.\\
  


  \end{abstract}
\section{Introduction}

Lie algebras of vector fields have been studied since the fundamental works of S. Lie and E. Cartan in the late 19th century and the early 20th century. A classical example of such Lie algebra is  the Lie algebra ${\mathcal W}_n$ consisting of the derivations of the polynomial algebra $\C[x_1,...,x_n]$, or, equivalently, the Lie algebra of polynomial vector fields on $\C^n$. The first classification results concerning representations of ${\mathcal W}_n$ and other Cartan type Lie algebras were  obtained by A. Rudakov in 1974-1975, \cite{Rud1}, \cite{Rud2}.  These results address the classification of a class of irreducible ${\mathcal W}_n$-representations that satisfy some natural topological conditions. The modules of Rudakov are a particular class of the so-called {\it tensor modules}.

General tensor modules $T(P,V)$ are introduced by Shen and Larson, \cite{Sh}, \cite{L}, and are defined for a $\mathcal D_n$-module $P$ and $\mathfrak{gl} (n)$-module $V$, where $\mathcal D_n$ is the algebra of polynomial differential operators on $\mathbb C^n$ (see \S 2.8 for details). The modules $T(P,V)$ have nice geometric interpretations. If $V$ is finite dimensional, then we have a natural map from ${\mathcal W}_n$ to the algebra of differential
operators in the section of a trivial vector bundle on $\C^n$ with fiber $V$. This map is a specialization of a Lie algebra homomorphism ${\mathcal W}_n\to \mathcal D_n\otimes U(\mathfrak{gl} (n))$. The tensor module
$T(P,V)$ is nothing but the pull back of the $\mathcal D_n\otimes U(\mathfrak{gl} (n))$-module $P\otimes V$.

Tensor  ${\mathcal W}_1$-modules and their extensions  were studied extensively in the 1970's and in the 1980's by B. Feigin, D. Fuks, I. Gelfand, and others, see for example, \cite{FF}, \cite{Fuk}.  Important results on general tensor modules $T(P,V)$ have been  recently established  by G. Liu, R. Lu, Y. Xue, K. Zhao, and others, see \cite{XL} and the references therein. 

In this paper we focus on the category of weight representations of ${\mathcal W}_n$, namely those that decompose as direct sums of weight spaces relative to the subalgebra $\mathfrak{h}$  of ${\mathcal W}_n$ spanned by the derivations $x_1\partial_1$,...,$x_n\partial_n$.   The study of weight representations of Lie algebras of vector fields is a subject of interest by both mathematicians and theoretical physicists in the last 30 years.  Two particular cases in this study have attracted special attention - the cases of $\mathcal W_n$ and of the Witt algebra ${\rm Witt}_n$. Recall that ${\rm Witt}_n$ is the Lie algebra of the derivations of the Laurent polynomial algebra $\C[x_1^{\pm},...,x_n^{\pm 1}]$, or, equivalently, the Lie algebra of polynomial vector fields on the $n$-dimensional complex torus. In particular, $\rm{Witt}_1$ is the centerless Virasoro algebra. The classification of all simple weight representations with finite weight multiplicities of  ${\mathcal W}_1$ and ${\rm Witt}_1$  (and hence of the Virasoro algebra) was obtained by O. Mathieu in 1992, \cite{M-Vir}. Following a sequence of works of S. Berman, Y. Billig,  C. Conley, X. Guo, C. Martin, O. Mathieu, V. Mazorchuk, V. Kac, G. Liu, R. Lu, A. Piard, S. Eswara Rao, Y. Su, K. Zhao,  recently, Y. Billig and V. Futorny managed to extend Mathieu's classification result to ${\rm Witt}_n$ for arbitrary $n\geq 1$ (see \cite{BF} and the references therein).\footnote{Note that the Witt algebra  ${\rm Witt}_n$ is denoted by $\mathcal W_n$ in \cite{BF}.} 

The classification of simple bounded (i.e. with a bounded set of weight multiplicities) modules of ${\mathcal W}_n$ was completed in \cite{XL}. The result in \cite{XL} states that every simple bounded module is a tensor module
$T(P,V)$  or a submodule of a tensor module. In order $T(P,V)$ to be bounded, $P$ must be a weight ${\mathcal D}_n$-module and $V$ must be a finite-dimensional module.

In this paper we classify all simple weight ${\mathcal W}_n$-modules $M$ with finite weight multiplicities. The main result is surprisingly easy to formulate - every such nontrivial module $M$ is either a tensor module $T(P,V)$ or the unique simple submodule of $T(P,\bigwedge\nolimits^k {\mathbb C}^n)$ for $k=1,...,n$. The necessary and sufficient condition for $P$ and $V$ so that  $T(P,V)$ has finite weight multiplicities is given in Theorem \ref{finmult}. This condition is expressed in terms of the subsets of roots ${\mathcal W}_n$ and $\mathfrak{gl} (n)$ that act locally finitely or injectively on $P$ and $V$, respectively. For our classification result, we first use a theorem of \cite{PS} stating that $M$ is parabolically induced from a bounded simple module $N$ over a subalgebra $\mathfrak g = {\mathcal W}_m\ltimes(\k\otimes \mathcal O_m)$ of ${\mathcal W}_n$. This subalgebra $\mathfrak g$ plays the role of a Levi subalgebra of a parabolic subalgebra of ${\mathcal W}_n$. The classification of simple bounded $\mathfrak g$-modules is one of the most difficult parts of the proof. By introducing the so called $(\g,\mathcal O_m)$-modules, we prove that $N$ is either the unique submodule of a tensor module, or it is a special generalized tensor module $\mathcal F (T(P,V), S)$, see Theorem \ref{mainparabolic}. The essential tool for proving this theorem is the twisted localization functor intrduced in \cite{M}. For the main theorem we show that the parabolic induction functor maps  $\mathcal F (T(P,V), S)$ to a tensor module.

The content of the paper is as follows. In Section 2 we collect some important definitions and preliminary results on weight modules, twisted localization, parabolic induction, and tensor modules. In Section 3 we prove  the necessary and sufficient condition for the tensor module $T(P,V)$ to be a weight module with finite weight multiplicities. We also show that $T(P,V)$ has a unique simple submodule and explain how the restricted duality functor acts on the tensor modules. The main theorem of this paper is also stated in Section 3. Section 4 is devoted to a few results concerning the parabolic induction theorem. The study of bounded $\g$-modules and the classification of all possible $\g$-modules $N$ that appear in the parabolic induction theorem are included in Section 5. In Section 6 we complete the proof of the main theorem by showing that the application of the parabolic induction functor on all possible $N$ described in the previous section leads to  modules $M$ that are either tensor modules or the unique simple submodules of $T(P,\bigwedge\nolimits^k {\mathbb C}^n)$ for $k=1,...,n$. 

\section{Preliminaries}
\subsection{Notation and convention}

Throughout the paper the ground field is $\mathbb C$. All vector spaces, algebras, and tensor products are assumed to be over $\mathbb C$ unless otherwise stated.

\subsection{Weight modules in general setting}\label{subsec-weight}

Let $\mathcal U$ be an associative unital algebra and $\mathcal H\subset\mathcal U$
be a commutative subalgebra. We assume in addition that 
$\mathcal H$  
is a polynomial algebra identified with the symmetric algebra of a vector space ${\mathfrak h}$, and that we
have a decomposition
$$\mathcal U=\bigoplus_{\mu\in {{\mathfrak h}^*}}\mathcal U^\mu,$$
where
$$\mathcal U^\mu=\{x\in\mathcal U | [h,x]=\mu(h)x, \forall h\in\mathfrak h\}.$$
Let $Q_{\mathcal U} = {\mathbb Z}\Delta_{\mathcal U} = \Delta_{\mathcal U} \cup (-\Delta_{\mathcal U} ) $ be the ${\mathbb Z}$-lattice in  ${\mathfrak h}^*$
generated by $\Delta_{\mathcal U}= \{ \mu \in {\mathfrak h}^* \; | \; {\mathcal U}^{\mu} \neq 0\}$. We  obviously have
${\mathcal U}^\mu {\mathcal U}^\nu\subset {\mathcal U}^{\mu+\nu}$. 

We call a ${\mathcal U}$-module $M$ {\it a weight module}, or a \emph{$({\mathcal U}, {\mathcal H})$-module},  if  $M = \bigoplus_{\lambda \in {\mathfrak h}^*} M^{\lambda}$, where  
$$
M^{\lambda} = \{m\in M \; | \; hm=\lambda(h)m\,\text{ for all } h \in {\mathfrak h}\}.
$$
We call $M^{\lambda}$ the  weight space of $M$, $\dim M^{\lambda}$ the \emph{$\lambda$-weight multiplicity} of $M$, and $\supp M = \{\lambda \in \h^* \; | \; M^{\lambda} \neq 0\}$ the \emph{support} of $M$. Note that 
\begin{equation*}
\mathcal U^\mu M^{\lambda}\subset M^{\mu+\lambda}.
\end{equation*}
for every weight module $M$.

We will call a weight $\mathcal U$-module \emph{bounded} if its set of weight multiplicities  is a bounded set. For a bounded $\mathcal U$-module $M$, the degree $d(M)$ is the maximal weight multiplicity of $M$. A weight  $\mathcal U$-module $M$ with finite weight multiplicities is \emph{cuspidal}
if all nonzero elements of ${\mathcal U}^{\mu}$ act injectively on $M$. If $\Delta_{\mathcal U} = - \Delta_{\mathcal U}$, then every cuspidal $\mathcal U$-module is bounded. We use this notion in the case when $\mathcal U$ is the
Weyl algebra or the universal enveloping algebra of a reductive Lie algebra where the latter property holds. 

In the particular case when 
${\mathcal U} = U(\g)$ for a Lie algebra $\mathfrak g$ and $\mathcal H =
S({\mathfrak h})$ for a Cartan subalgebra $\mathfrak h$ of $\mathfrak g$, we have that a  weight  ${\mathcal U}$-module is a  weight $\g$-module.

\subsection{Twisted localization in general setting} \label{subsec-tw-loc} We retain the notation of the previous subsection. 

Let  $a$ be an ad-nilpotent element of $\mathcal U$. Then the set $\langle a \rangle = \{ a^n \; | \; n \geq 0\}$ is an Ore subset of $\mathcal U$  which allows us to define the  $\langle a \rangle$-localization $D_{\langle a \rangle} \mathcal U$ of $\mathcal U$. For a $\mathcal U$-module $M$  by $D_{\langle a \rangle} M = D_{\langle a \rangle} {\mathcal U} \otimes_{\mathcal U} M$ we denote the $\langle a \rangle$-localization of $M$. Note that if $a$ is injective on $M$, then $M$ is isomorphic to a submodule of $D_{\langle a \rangle} M$. In the latter case we will identify $M$ with that submodule.

We next recall the definition of the generalized conjugation of $D_{\langle a \rangle} \mathcal U$ relative to $x \in {\mathbb C}$. This is the automorphism  $\phi_x : D_{\langle a \rangle} \mathcal U \to D_{\langle a \rangle} \mathcal U$ defined by the formula
$$\phi_x(u) = \sum_{i\geq 0} \binom{x}{i} \ad (a)^i (u) a^{-i}.$$ 
If $x \in \mathbb Z$, then $\phi_x(u) = a^xua^{-x}$. With the aid of $\phi_x$ we define the twisted module $\Phi_x(M) = M^{\phi_x}$ of any  $D_{\langle a \rangle} \mathcal U$-module $M$. Finally, we set $D_{\langle a \rangle}^x M = \Phi_x D_{\langle a \rangle} M$ for any $\mathcal U$-module $M$ and call it the \emph{twisted localization} of $M$ relative to $a$ and $x$. We will use the notation $a^x\cdot m$  (or simply $a^x m$) for the element in  $D_{\langle a \rangle}^x M$ corresponding to $m \in D_{\langle a \rangle} M$. In particular, the following formula holds in $D_{\langle a \rangle}^{x} M$:
$$
u (a^{x} m) = a^{x} \left( \sum_{i\geq 0} \binom{-x}{i} \ad (a)^i (u) a^{-i}m\right)
$$
for $u \in \mathcal U$, $m \in D_{\langle a \rangle}  M$.

If $a_1,...,a_k$ are commuting ad-nilpotent elements  in $\mathcal U$ and ${\bf c} = (c_1,...,c_k)$ is in $\C^k$, then we set $ D_{\langle a_1,...,a_k \rangle} M = \prod_{i=1}^k D_{\langle a_i \rangle}M$ and  $ D_{\langle a_1,...,a_k \rangle}^{\bf c} M = \prod_{i=1}^k D_{\langle a_i \rangle}^{c_i} M$. Note that the products  $ \prod_{i=1}^k D_{\langle a_i \rangle}$ and $ \prod_{i=1}^k D_{\langle a_i \rangle}^{c_i}$ are well defined because the functors involved pairwise commute.

If $a\in\mathcal U$ is an ad-nilpotent weight element and $M$ is a weight module then $D_{\langle a \rangle}^x M$ is again a weight module.

\begin{lemma} \label{lem-tw-loc-simple}
Let $a \in {\mathcal U}$ be an $\ad$-nilpotent weight element in ${\mathcal U}$, $M$ be a simple $a$-injective weight ${\mathcal U}$-module, and  $z \in \C$. If $N$ is a simple nontrivial submodule of $D^z_{\langle a \rangle}M$, then $D_{\langle a \rangle}M \simeq D^{-z}_{\langle a \rangle}N$. In particular, if $a$ acts bijectively on $M$, $M  \simeq D^{-z}_{\langle a \rangle}N$.
\end{lemma}
\begin{proof}
We use the fact  that if $M$ is a simple weight ${\mathcal U}$-module, then $D_{\langle a \rangle}M$ and $D_{\langle a \rangle}^z M$  are simple $D_{\langle a \rangle} {\mathcal U}$-modules. Since $N$ is submodule of $D^z_{\langle a \rangle}M$,  $D_{\langle a \rangle} N$ is a submodule of $D^z_{\langle a \rangle}M$. The simplicity of $N$ implies  $D_{\langle a \rangle} N \simeq D^z_{\langle a \rangle}M$ and $D^{-z}_{\langle a \rangle} N \simeq D_{\langle a \rangle}M$. If $a$ acts bijectively, then $M \simeq D_{\langle a \rangle}M$. \end{proof}

We will also consider the following particular case of the twisted localization functor for $\mathcal U = U(\mathfrak g)$, where $\g={\mathcal W}_m+\k \otimes\mathcal O_m$. Let $a_i \in \k^{\alpha_i}$, $i=1,...,\ell$, and $\Gamma = \{ \alpha_1,...,\alpha_{\ell}\}$ is a set of commuting roots of $\k$ that is linearly independent in $\mathbb Z \Delta_{\k}$.  Let also $\lambda \in \h^*$ be such that $\lambda = \sum_{i=1}^{\ell} z_i \alpha_i$. We set $D_{\Gamma}^{\lambda} = D_{\langle a_1 \rangle}^{z_1}...D_{\langle a_\ell \rangle}^{z_\ell}$. If $M \simeq D_{\Gamma}^{\lambda} \bar M$ we will say that $M$ is \emph{obtained by a twisted localization from $\bar M$}. 
If $\bar M$ is bounded, then $M$ is bounded, \cite{M} Lemma 4.4.

\subsection{The algebras $\mathcal O_n$, $\mathcal D_n$, and ${\mathcal W}_n$} \label{subsec-o-d-w}

In what follows, $\mathcal O_n = \mathbb C [x_1,...,x_n]$ and ${\mathcal D}_n$ will stand for the associative algebra of differential operators in $\mathcal O_n$.  In other words, ${\mathcal D}_n$ is the $n$-th Weyl algebra. We will often use the fact that ${\mathcal D}_n \simeq  {\mathcal D}_1 \otimes ... \otimes {\mathcal D}_1$ ($n$ copies).  Also, ${\mathcal W}_n$ will stand for the Lie algebra of vector fields on $\mathbb C^n$, i.e. ${\mathcal W}_n = \mbox{Der} (\mathcal O_n)$.

Henceforth, we fix $\h = \mbox{Span} \{ t_1\partial_1,\dots,t_n\partial_n\}$. Note that $\h$ is a Cartan subalgebra of ${\mathcal W}_n$ and  $\mathcal H=\mathbb C[t_1\partial_1,\dots,t_n\partial_n]$ is a maximal commutative subalgebra of $\mathcal D_n$. We will use the setting of \S \ref{subsec-weight} and \S \ref{subsec-tw-loc} both  for $\mathcal U = U({\mathcal W}_n)$ and $\mathcal U = \mathcal D_n$, and in both cases $\h$ is the one that we fixed above. The set of roots of ${\mathcal W}_n$ is: 
$$
\Delta = \left\{\sum_{j=1}^n m_j\varepsilon_j,  -\varepsilon_i+ \sum_{j\neq i} m_j\varepsilon_j \mid m_j \in \mathbb Z_{\geq 0}, i=1,...,n\right\},
$$
where $\varepsilon_i (t_j \partial_j) = \delta_{ij}$. By $\Delta'$ we denote the set of all invertible roots of ${\mathcal W}_n$. One can see that $\Delta'$ is a root system of type $A_{n}$.

A ${\mathcal W}_n$-module $M$ is a \emph{$({\mathcal W}_n,\mathcal O_n)$-module} if
  $M$ is an $\mathcal O_n$-module satisfying
  \begin{equation*}
    X(fv)=fX(v)+X(f)v,\ \  \forall v\in M, f\in\mathcal O_n, X\in {\mathcal W}_n.
  \end{equation*}
 If $M$ is a weight ${\mathcal W}_n$-module with finite weight multiplicities, then the \emph{restricted dual} $M_*$ of $M$ is by definition the maximal semisimple $\h$-submodule of $M^*$.
 The following properties of the restricted dual functor are straightforward.
 \begin{lemma}\label{restricted_dual} Let $M$ be a weight ${\mathcal W}_n$-module with finite weight multiplicities. Then
   \begin{enumerate}
   \item $\supp M_*=-\supp M$;
   \item $\dim M_*^\mu=\dim M^{-\mu}$;
     \item $M$ is simple if and only if $M_*$ is simple.
     \end{enumerate}
   \end{lemma}

   Consider the embedding $\mathbb C^n\to \mathbb CP^n$. The Lie algebra of vector fields on  $\mathbb CP^n$ is isomorphic to $\mathfrak{sl}(n+1)$  and is a Lie subalgebra of ${\mathcal W}_n$.  In other words we have a canonical embedding
   $\mathfrak{sl}(n+1)\subset {\mathcal W}_n$ of Lie algebras.
   
    \begin{lemma}\label{bnd-wn-finite-length} Let $M$ be a bounded weight ${\mathcal W}_n$-module such that $\supp M \subset \lambda + \mathbb Z \Delta_{{\mathcal W}_n}$ for some weight $\lambda$. Then $M$ has finite length. 
       \end{lemma}
       \begin{proof}
        The result holds for $\mathfrak{sl}(n+1)$-modules, see Lemma 3.3 in \cite{M}, and hence it holds for ${\mathcal W}_n$ by using the natural embedding of $\mathfrak{sl}(n+1)$ in ${\mathcal W}_n$.
       \end{proof}

\subsection{Simple weight $\mathcal D_n$-modules} \label{subsec-simple-weight} According to \S  \ref{subsec-weight}, a ${\mathcal D}_n$-module $M$ is a {weight module} if 
$$
M = \bigoplus_{\lambda \in \C^n}M^{\lambda},
$$
where $M^{\lambda} = \left\{ m \in M \; | \; x_i\partial_i m = \lambda_i m, \mbox{ for } i =1,...,n \right\}$.
Below we recall the classification of the simple weight $\mathcal D_n$-modules.

We will use the automorphism $\sigma_F : {\mathcal D}_n\to {\mathcal D}_n$  defined by
$\sigma_F(t_i)=\partial_i, \sigma_F(\partial_i)=-t_i$ for all $i$.
We call  $\sigma_F$ the (full) \emph{Fourier transform} of  ${\mathcal D}_n$. If $M$ is a ${\mathcal D}_n$-module, by $M^F$ we denote the module $M$ twisted by $\sigma_F$.  

The following gives the classification of all simple weight ${\mathcal D}_n$-modules, see for example Corollary 2.9 in \cite{GS1} 
\begin{proposition}
\begin{itemize}
\item[(i)] Every simple weight module of $ {\mathcal D}_1$ is isomorphic to one of the following: $\mathcal O_1 = \C [x]$, $\mathcal O_1^F$, $x^{\lambda} \C [x^{\pm 1}]$, $\lambda \in \C \setminus \mathbb Z$.
\item[(ii)] Every simple weight module of $ {\mathcal D}_n$ is isomorphic to $P_1 \otimes ...\otimes P_n$ where $P_i$ is a simple weight $ {\mathcal D}_1$-module.
\end{itemize}
\end{proposition}
We note also that every simple nontrivial weight $\mathcal D_n$-module $M$  has degree $1$, i.e. all its weight multiplicities equal $1$. Moreover, for every $i$, $x_i$ (respectively, $\partial_i$) acts either injectively, or locally nilpotently on $M$.
Let $I^+(M)$ denote the subset of indices in $\{1,\dots,n\}$ such that $\partial_i$ acts locally nilpotently on $M$ and $x_i$ acts injectively on $M$,
        $I^-(M)$ the subset of indices such that $x_i$ acts locally nilpotently on $M$ and $\partial_i$ acts injectively on $M$, and  $I^0(M)$ the subset of indices such that both $x_i$ and $\partial_i$ act injectively on $M$. Note that $\{1,\dots,n\} = I^-(M) \sqcup I^0(M) \sqcup I^+(M)$. Furthermore, there exists $\lambda \in \supp M$ such that
$$
\supp M = \lambda + \sum_{i \in I^+(M)} \mathbb Z_{\geq 0} \varepsilon_i + \sum_{j \in I^0(M)} \mathbb Z \varepsilon_j +  \sum_{k \in I^-(M)} \mathbb Z_{\leq 0} \varepsilon_k.
$$

 \subsection{Parabolic induction in general}
Let $\g$ be any Lie algebra with Cartan subalgebra $\h$ such that $\g=\h\oplus\bigoplus_{\alpha\in\Delta} \g_\alpha$. Let $\gamma\in\h^*$. Then the subalgebra
$$\p=\h\oplus\bigoplus_{Re\langle \gamma,\alpha\rangle\geq 0}\g_\alpha$$
is called the \emph{parabolic subalgebra} of $\g$ corresponding to $\gamma$. The Levi subalgebra of $\p$ is 
$$\l=\h\oplus\bigoplus_{Re\langle \gamma,\alpha\rangle=0 }\g_\alpha,$$
and the nilradical of $\p$ is
$$\n=\bigoplus_{Re\langle \gamma,\alpha\rangle>0}\g_\alpha,$$

We are going to use extensively the following standard result.
\begin{proposition}\label{parabolicind}
  
(a)  Let $N$ be a simple $\l$-module, considered also as simple $\p$-modules by letting $\n$ act trivially on $N$. Then the $\g$-module $U(\g)\otimes_{U(\p)}N$
has a unique simple quotient.

(b) If $L$ is a simple $\g$-module such that $L^\n\neq 0$, then $L^\n$ is a simple $\l$-module.

(c) If $L$ and $M$ are simple $\g$-modules such that $M^\n\simeq L^\n$ as $\l$-modules, then $M \simeq L$ as $\g$-modules.
\end{proposition}
\begin{remark} \label{rk-top}If $M$ is a simple weight $\g$-module then $M^\n=\bigoplus_{\lambda\in S}M^\lambda$ where $S$ is the subset of $\supp M$ such that $\lambda+\alpha\notin \supp M$ for any $\alpha\in\Delta(\n)$.
 For an arbitrary weight module $M$ we call $\oplus_{\lambda\in S}M^\lambda$ the \emph{$\p$-top} of $M$ and denote it by $M^{\rm top}$.
  \end{remark}

 \subsection{Parabolic induction for ${\mathcal W}_n$} \label{subsec-par-ind}

In this subsection we recall  one of the main results in \cite{PS}. Recall the definitions of  $\Delta$ and $\Delta'$ from \S \ref{subsec-o-d-w}.
Let $\gamma=a_1\varepsilon_1+\dots+a_n\varepsilon_n$ for some $a_i\in\mathbb R$.
Set $$\Delta_0=\{\alpha\in\Delta\mid(\gamma,\alpha)=0\},\quad \Delta_{\pm}=\{\alpha\in\Delta\mid(\gamma,\alpha)>(<0)\},$$
$$\Delta'_0=\Delta_0\cap\Delta',\quad \Delta'_\pm=\Delta_\pm\cap\Delta'.$$
Let $$\p=\h\oplus\bigoplus_{\alpha\in\Delta_0\cup\Delta_+} ({\mathcal W}_n)_\alpha,\quad \g=\h\oplus\bigoplus_{\alpha\in\Delta_0} ({\mathcal W}_n)_\alpha.$$

\begin{theorem}\label{PS} \label{PS}Let $M$ be a simple weight ${\mathcal W}_n$-module. 

  (a) There exists a weight $\lambda\in\supp M$ and $\gamma$ such that $$\supp M\subset \lambda+{{\mathbb Z}_{\geq 0}}(\Delta_-'\cup\Delta'_0).$$

  (b) One can choose $\gamma$ in such a way that $\mathbb Z\Delta_0'=\mathbb Z\Delta_0$ and $$\lambda+\mathbb Z\Delta_0\subset\supp M.$$

  (c) $M$ is a unique simple quotient of the parabolically induced module $U({\mathcal W}_n)\otimes_{U(\p)}M_0$ for some simple weight $\g$-module $M_0$ that is extended in the natural way to a simple $\p$-module.
\end{theorem}

\subsection{Tensor modules over ${\mathcal W}_n$}\label{subsec-tensor-modules}
Let $V$ be a $\mathfrak{gl} (n)$-module and $\tilde V:=\mathcal O_n\otimes V$. One can look at $\tilde V$ as the space of sections of the $\mathfrak{gl}(n)$-bundle on $\mathbb C^n$ with fiber $V$.
Thus, $\tilde V$ has the natural structure of a $({\mathcal W}_n,\mathcal O_n)$-module. 
   
For a ${\mathcal D}_n$-module $P$ and a $\mathfrak{gl} (n)$-module $V$, we define the tensor $({\mathcal W}_n, \mathcal O_n)$-module by $$T(P,V) := P\otimes_{\mathcal O_n}\tilde V$$ and call it \emph{the tensor ${\mathcal W}_n$-module relative to
$P$ and $V$}. If $P$ is a weight ${\mathcal D}_n$-module and $V$ is a weight $\mathfrak{gl}(n)$-module then $P\otimes_{\mathcal O_n}\tilde V$ is a weight module and
        $$\supp (P\otimes_{\mathcal O_n}\tilde V)=\supp P+\supp V.$$

    Alternatively, we can define $T(P,V)$ as follows. Consider  $T(P,V)$ as the vector space $T(P,V) = P\otimes_{\mathbb C} V$   and define  $\mathcal W_n$-action and $\mathcal O_n$-action by the formulas
        \begin{eqnarray*}
        x^{\alpha} \partial_j \cdot (f \otimes v) & = &  x^{\alpha} \partial_j f \otimes v + \sum_{i=1}^n \partial_i (x^{\alpha})f \otimes E_{ij}v, \\
         x^{\alpha} \cdot (f \otimes v) & = & x^{\alpha}f \otimes v,
        \end{eqnarray*}
        for $f \in P$, $v \in V$.

In what follows, the $k$-th exterior power $\bigwedge^k {\mathbb C}^n$ of the natural representation of $\mathfrak{gl(n)}$ will be called the $k$-th fundamental representation. We have the following result from  \cite{LLZ} (Theorem 3.1 and Lemma 3.7):

\begin{proposition} \label{prop-tensor-simple}
\begin{itemize}
\item[(i)] Let $P$ be a simple ${\mathcal D}_n$-module and $V$ be a simple $\mathfrak{gl}(n)$-module that is not isomorphic to a fundamental representation. Then $T(P,V)$ is a simple ${\mathcal W}_n$-module.
\item[(ii)]
Let $P_1$ and $P_2$ be simple  ${\mathcal D}_n$-modules  
and let $V_1$ and $V_2$ be simple $\mathfrak{gl}(n)$-modules such that neither of them is isomorphic to a fundamental representation. Then $T(P_1,V_1) \simeq T(P_2,V_2)$ if and only if $P_1 \simeq P_2$ and $V_1 \simeq V_2$.
\end{itemize}
\end{proposition}

We next consider tensor modules $T(P,V)$ for which $V$ is a fundamental representation. For any ${\mathcal D}_n$-module $P$, the \emph{differential map}
\begin{equation*} 
d : T(P, \bigwedge {\mathbb C}^n) \to T(P, \bigwedge {\mathbb C}^n),
\end{equation*}
is defined by $d(f\otimes v) = \sum_{i=1}^n (\partial_i f) \otimes (e_i \wedge v)$, where $(e_1,...,e_n)$ is the standard basis of ${\mathbb C}^n$ associated to the coordinates $x_1,...,x_n$ of  ${\mathbb C}^n$. The map $d$ is a homomorphsim of ${\mathcal W}_n$-modules but not $\mathcal O_n$-modules. One readily sees
that $d^2=0$. As a result we have the following generalized de Rham complex:
$$
0 \xrightarrow[]{d}T(P, \bigwedge\nolimits^0 {\mathbb C}^n) \xrightarrow[]{d}  T(P, \bigwedge\nolimits^1 {\mathbb C}^n) \xrightarrow[]{d} \cdots \xrightarrow[]{d} T(P, \bigwedge\nolimits^n {\mathbb C}^n) \xrightarrow[]{d}  0.
$$

By Theorem 3.5 in \cite{LLZ} we have the following.

\begin{proposition}\label{tensormodule} Let $P$ be a simple $\mathcal D_n$-module.
\begin{itemize}

\item[(i)] If $k = 0,...,n-1$, then the module $T(P, \bigwedge\nolimits^k {\mathbb C}^n)$ has a simple quotient isomorphic to $d T(P, \bigwedge\nolimits^k {\mathbb C}^n)$. 

\item[(ii)] The module $T(P, \bigwedge\nolimits^0 {\mathbb C}^n)$ is simple if and only if $P$ is not isomorphic to $\mathcal O_n$. If $P\simeq \mathcal O_n$ then $T(P, \bigwedge\nolimits^0 {\mathbb C}^n)$ contains a trivial
  ${\mathcal W}_n$-submodule $\mathbb C$ and
  $dT(P, \bigwedge\nolimits^{0} {\mathbb C}^n)\simeq T(P, \bigwedge\nolimits^{0} {\mathbb C}^n)/\mathbb C$.

\item[(iii)]The module $T(P, \bigwedge\nolimits^n {\mathbb C}^n)$ is simple if and only if $\sum_i \partial_i P \neq P$.
\end{itemize}
\end{proposition}

We finish this subsection by stating the main result from \cite{XL} concerning the classification of the simple bounded ${\mathcal W}_n$-modules.

\begin{theorem}\label{thm-class-bnd}  Let  $M$ be a nontrivial simple bounded ${\mathcal W}_n$-module. Then $M$  is isomorphic to one of the following:
\begin{itemize}
\item[(a)] the module $T(P, V )$, where $P$ is a simple weight $\mathcal D_n$ module and $V$ is
a simple finite-dimensional $\mathfrak{gl}(n)$-module that is not isomorphic to a fundamental representation;
\item[(b)] a simple submodule of  $T(P,  \bigwedge\nolimits^k {\mathbb C}^n)$, where $k \in \{1,2,...,n \}$, and $P$ is a simple weight ${\mathcal D}_n$ module.
\end{itemize}
\end{theorem}

\begin{remark} Proposition \ref{tensor_sub} implies that $T(P,  \bigwedge\nolimits^k {\mathbb C}^n)$ has a unique simple submodule.
  \end{remark}

\section{Tensor modules with finite weight multiplicities}

\subsection{Tensor product of weight $\mathfrak{gl}(n)$-modules}

        Let $M$ be a simple weight $\mathfrak{gl}(n)$-module with finite weight multiplicities. Recall from \cite{DMP} that $M$ has the following \emph{shadow decomposition}:
        $$\Delta(\mathfrak{gl}(n))=\Delta^F_M\sqcup \Delta^I_M\sqcup \Delta^+_M\sqcup\Delta^-_M,$$
        such that the $\alpha$-root vectors $X_\alpha$ act locally nilpotently on $M$ for all roots $\alpha\in \Delta^+_M\sqcup\Delta^F_M$ and injectively 
        for all roots $\alpha\in \Delta^-_M\sqcup\Delta^I_M$. Moreover, $\Delta^I_M\sqcup\Delta^F_M$ and $ \Delta^+_M$ are the roots of the Levi subalgebra $\g^I+\g^F$ and the nilradical $\g^+$, respectively, of a parabolic
        subalgebra $\p\subset\mathfrak{gl}(n)$,
        and $M$ is a quotient a parabolically induced module $\operatorname{Ind}_\p^{\mathfrak{gl}(n)}\left( M^F\otimes M^I \right)$, for some cuspidal simple $\g^I$-module $M^I$ and some finite-dimensional simple $\g^F$-module $M^F$.

        \begin{lemma}\label{tensor-parabolic} Let $\l$ be the Levi subalgebra of some parabolic $\p$ in $\mathfrak{gl}(n)$. Assume that $M'$ and $N'$ are weight $\l$-modules and that $M'\otimes N'$ has finite weight multiplicities.
          Then $(\operatorname{Ind}_\p^{\mathfrak{gl}(n)}M')\otimes(\operatorname{Ind}_\p^{\mathfrak{gl}(n)}N')$ has finite weight multiplicities.
        \end{lemma}
        \begin{proof} Let $\m$ denote the nilradical of the opposite to $\p$ parabolic subalgebra $\p^-$, and let $U=U(\m)$. Then $U$ has a  $\mathbb Z_{\geq 0}$-grading $U=\bigoplus_{p\geq 0}U_p $  such that $U_0=\mathbb C$ and each $U_p$ is a finite-dimensional $\l$-module. This grading induces
         $\mathbb Z_{\geq 0}$-gradings
          on both $M=\operatorname{Ind}_\p^{\mathfrak{gl}(n)}M'$ and $N=\operatorname{Ind}_\p^{\mathfrak{gl}(n)}N'$ so that
          $M_p=M'\otimes U_p$ and $N_p=N'\otimes U_p$. Then $M\otimes N$ is also graded and its $m$th  graded component is 
                    $$(M\otimes N)_m=\bigoplus_{p+q=m}M'\otimes N'\otimes U_p\otimes U_q.$$
          Hence, $M\otimes N$  has finite weight multiplicities.
          \end{proof}
          \begin{lemma}\label{gl} Let $M$ and $N$ be simple weight $\mathfrak{gl}(n)$-modules. Then $M\otimes N$ has finite weight multiplicities if and only if $(\Delta^I_M\sqcup\Delta^-_M)\subset(\Delta^F_N\sqcup \Delta^-_N)$
            or, equivalently,  $(\Delta^I_M\sqcup\Delta^-_M)\cap(\Delta^I_N\sqcup \Delta^+_N)=\emptyset$.
        \end{lemma}
        \begin{proof} First assume that the condition is not true. There exists a root $\alpha\in \Delta^I_M\sqcup\Delta^-_M$ such that $-\alpha\in \Delta^I_N\sqcup\Delta^-_N$. If $\mu\in\supp M$ and
          $\nu\in \supp N$ then $\mu+{{\mathbb Z}_{\geq 0}}\alpha\subset\supp M$
          and  $\nu-{{\mathbb Z}_{\geq 0}}\alpha \in\supp N$. Hence $\mu+\nu$ has infinite multiplicity in $M\otimes N$.

          Next assume that the condition holds. Then $\Delta^I_M\subset \Delta^F_N$, $\Delta^I_N\subset \Delta^F_M$ and $\Delta_M^-\subset(\Delta^F_N\sqcup \Delta^-_N)$. Choose $\gamma_M\in \mathbb Q\Delta$ such that
          $(\gamma_M,\alpha)=0$ for all $\alpha\in\Delta_M^I\sqcup\Delta_M^F$ and $(\gamma_M,\alpha)<0$ for all $\alpha\in \Delta_M^-$. Similarly choose $\gamma_N$, and let $\gamma=\gamma_M+\gamma_N$.
          Then $(\gamma,\Delta_M^I)=(\gamma,\Delta_N^I)=0$ and $(\gamma,\alpha)<0$ for any $\alpha\in \Delta^-_M\cup\Delta^-_N$. Let $\p$ be the parabolic defined by $\gamma$. Then both $M$ and $N$ are quotients of the parabolically induced
          modules $\operatorname{Ind}_\p^{\mathfrak{gl}(n)}M'$ and $\operatorname{Ind}_\p^{\mathfrak{gl}(n)}N'$, respectively. The Levi subalgebra $\l$ of $\p$ is isomorphic to $\g^I_M\oplus \g^I_N\oplus (\g^F_M\cap\g^F_N)$.
          Furthermore, $M'=M^i\otimes M^f$ where $M^i$ is a simple cuspidal $\g^I_M$-module and $M^f$ is some finite-dimensional $\g^I_N\oplus (\g^F_M\cap\g^F_N)$-module. Similarly,
          $N'=N^i\otimes N^f$ where $N^i$ is a simple cuspidal $\g^I_N$-module and $N^f$ is some finite-dimensional $\g^I_M\oplus (\g^F_M\cap\g^F_N)$-module. Therefore $M'\otimes N'$ has finite weight multiplicities and
          the statement follows from Lemma \ref{tensor-parabolic}.
        \end{proof}
        
\subsection{Weight tensor modules}
          \begin{lemma}\label{restriction} Let $P$ be a simple weight ${\mathcal D}_n$-module. Then $P=\bigoplus_{\kappa} P_\kappa$, where $P_\kappa$ is the eigenspace of
          $\sum_{i=1}^n x_i\partial_i$ with eigenvalue $\kappa$. Furthermore,  every nonzero $P_\kappa$ is a simple $\mathfrak{gl}(n)$-module and all nonzero $P_\kappa$ have the same shadow.
        \end{lemma}
        \begin{proof} The first assertion is obvious. Since the adjoint action of $\mathfrak{gl}(n)$ on $\mathcal D_n$ is locally finite, every root vector $X_\alpha\in\mathfrak{gl}(n)$ either acts locally nilpotently or
          injectively on all nonzero vectors of $P$.
          Therefore all $P_\kappa$ have the same shadow. By the classification of simple weight $\mathcal D_n$-modules,  every $P_\kappa$ is multiplicity free and $\supp P_\kappa \subset \lambda+\mathbb Z\Delta(\mathfrak{gl}(n))$ for any weight $\lambda \in \supp P_\kappa$. Using these and the fact that $U(\mathfrak{gl}(n)) P_{\kappa}^{\lambda} = P_{\kappa}$, we obtain that $P_\kappa$ is simple.
        \end{proof}
        \begin{remark} \label{rem-shadow}
        Lemma \ref{restriction} implies that every simple $\mathcal D_n$-module $P$ has a well-defined $\mathfrak{gl}(n)$-shadow. Below we give an explicit description of this shadow in terms of the 
        subsets $I^{\pm}(P)$, $I^0(P)$ of $\{1,\dots,n\}$ defined in \S\ref{subsec-simple-weight}:
         $$\Delta^I_P=\{\varepsilon_i-\varepsilon_j\mid i,j\in I^0(P)\},\quad \Delta_{P}^F=\{\varepsilon_i-\varepsilon_j\mid i,j\in I^+(P)\,\text{or}\,i,j\in I^-(P)\},$$
       $$\Delta_{P}^{-}=\{\varepsilon_i-\varepsilon_j\mid i\in I^+(P),j\notin I^+(P)\,\text{ or }\,i\notin I^-(P), j\in I^-(P)\},\quad \Delta_P^+=-\Delta_P^-.$$
      \end{remark}
      
        \begin{theorem}\label{finmult} Let $P$ be a simple weight $\mathcal D_n$-module and $V$ be a  simple weight $\mathfrak{gl}(n)$-module. Then     
 the ${\mathcal W}_n$-module $T(P,V)$ has finite weight multiplicities
          if and only if  $(\Delta^I_P\sqcup\Delta^-_P)\subset(\Delta^F_V\sqcup \Delta^-_V)$.
        \end{theorem} 
        \begin{proof} For every semisimple $\h$-module $X$ we denote by $X_\kappa$ the eigenspace of $\sum x_i\partial_i$ with eigenvalue $\kappa$. By Lemma  \ref{restriction},  $P=\bigoplus_{\tau\in\tau_0+\mathbb Z} P_\tau$ for
          some $\tau_0\in\mathbb C$. Then
          $$T(P,V)=\bigoplus_{\tau\in\tau_0+\mathbb Z} P_{\tau}\otimes V,$$
          and the statement follows from Lemma \ref{gl}.
        \end{proof}

              \begin{example}
                Consider a simple highest weight module $\mathfrak{gl} (4)$-module $V$ such that
                  $$\Delta_V^{F} \sqcup   \Delta_V^{+} = \left\{\varepsilon_i - \varepsilon_j  \; | \; i=1,2; j=3,4 \right\}.$$  Let $P$ be a simple weight ${\mathcal D}_n$-module $P$ on which $x_1,x_3, \partial_2,\partial_4$ act
                  injectively and $\partial_1,\partial_3, x_2,x_4$ act locally nilpotently. Then by Remark \ref{rem-shadow} and Theorem \ref{finmult}, $T(P,V)$ has infinite weight multiplicities as $\varepsilon_1-\varepsilon_4\in \Delta^-_P\cap\Delta^+_V$.
                  On the other hand, if $P'$ is a simple ${\mathcal D}_n$-module on which $x_1, \partial_2,\partial_3, \partial_4$ act locally nilpotently and $\partial_1, x_2, x_3, x_4$ act injectively,
                  then $T(P',V)$ has finite weight multiplicities.
              \end{example}
        \begin{proposition}\label{tensor_sub} For any simple weight $\mathcal D_n$-module $P$ and any simple weight $\mathfrak{gl}(n)$-module $V$, the ${\mathcal W}_n$-module $T(P,V)$ has a unique simple submodule. 
\end{proposition}
\begin{proof} If $V$ is not a fundamental representation the statement follows from Proposition \ref{prop-tensor-simple}(i). Now let $V=\bigwedge\nolimits^{k} {\mathbb C}^n$.
  It is shown in \cite{GS2} that if $P$ is cuspidal, i.e., $I^+(P)=I^-(P)=\emptyset$, then $T(P, \bigwedge\nolimits^k {\mathbb C}^n)$ is simple  for $k=0,n$ and 
   an indecomposable   $\mathfrak{sl}(n+1)$-module of length two
 for $k=1,\dots, n-1$.
  This implies the statement for  a cuspidal module $P$. For a general module $P$, consider
  $$\gamma=s\sum_{i\in I^-(P)}\varepsilon_i-\sum_{j\in I^+(P)}\varepsilon_j$$
  for some irrational $s>1$. Let $\p$ be the corresponding parabolic subalgebra of ${\mathcal W}_n$ and $\mathfrak n$ be the nilradical of $\p$. The Levi subalgebra $\g$ is isomorphic to $\mathfrak{gl}(p)\oplus \mathfrak{gl}(q)\oplus {\mathcal W}_m$
  where $p=|I^-(P)|$,  $q=|I^+(P)|$,  and $m=|I^0(P)|$. Note that
  $$P\simeq \mathcal O_p^F\otimes\mathcal O_q\otimes P_m$$
  for some cuspidal ${\mathcal D}_m$-module $P_m$. Since $V$ is finite dimensional and simple,  $V^{\mathfrak n\cap\mathfrak{gl}(n)}\simeq V_p\otimes V_q\otimes V_m$ is a simple module over $\mathfrak{gl}(p)\oplus \mathfrak{gl}(q)\oplus \mathfrak{gl}(m)$. It is easy to compute that
  $$T(P,V)^{\mathfrak n}\simeq V_p\otimes V_q\otimes T(P_m,V_m).$$
  Since $P_m$ is cuspidal, $ T(P_m,V_m)$ has a unique simple ${\mathcal W}_m$-submodule and hence $T(P,V)^{\mathfrak n}$ has a unique simple $\g$-submodule $N$. If $M$ is a simple ${\mathcal W}_n$ submodule of $T(P,V)$ then $M^{\mathfrak n}\neq 0$ and hence
  $N\subset M$. That implies the uniquness of $M$.
\end{proof}      
        
        \subsection{Duality for tensor modules}

        \begin{lemma}\label{dualDmod} Let $P$ be a simple weight ${\mathcal D}_n$-module. Consider $P$ as a ${\mathcal W}_n$-module via the natural homomorphism ${\mathcal W}_n\to {\mathcal D}_n$. Then $P_*\simeq T(P^F,\Lambda^n\mathbb C^n)$.
        \end{lemma}
        \begin{proof} Recall the definition of $I^\pm(P), I^0(P)$. As a vector space
          $$P=\prod_{i\in I^0(P)}x_i^{\lambda_i}\otimes\mathbb C[x_j]_{j\in I^+(P)}\otimes \mathbb C[\partial_k]_{k \in I^-(P)}\otimes \mathbb C[x^{\pm 1}_\ell]_{\ell\in I^+(P)},$$
          where  $\lambda_i$ are  nonintegral for all $i \in I^0(P)$.
          We denote the monomial basis of $P$ by $e(\mu)$ where $\mu_i\in\lambda_i+\mathbb Z$ for $i\in I^0(P)$, $\mu_i\in\mathbb Z_{\geq 0}$ for $i\in I^+(P)\sqcup I^-(P)$.
          We have
          $$x_i e(\mu)=\begin{cases}e(\mu+\varepsilon_i)\ \text{if}\ i\in I^0(P)\cup I^+(P)\\ -\mu_ie(\mu+\varepsilon_i)\ \text{if}\ i\in  I^-(P)\end{cases},$$
          $$\partial_i e(\mu)=\begin{cases}\mu_ie(\mu-\varepsilon_i)\ \text{if}\ i\in I^0(P)\cup I^+(P)\\ e(\mu+\varepsilon_i)\ \text{if}\ i\in I^-(P)\end{cases}.$$
          Denote the corresponding basis of $P^F$ by $f(\mu)$ where $\mu$ runs over the same set as $e(\mu)$. Using identification of $P$ and $P^F$ as vector spaces, we have that if $X(e(\mu))=ce(\nu)$ then $\sigma_F(X)f(\mu)=cf(\nu)$.
          This  observation allows us to write the action of generators in the basis  $f(\mu)$:
  $$\partial_i f(\mu)=\begin{cases}-f(\mu+\varepsilon_i)\ \text{if}\ i\in I^0(P)\cup I^+(P),\\ \mu_if(\mu-\varepsilon_i)\ \text{if}\ i\in  I^-(P),\end{cases}$$
          $$x_i f(\mu)=\begin{cases}\mu_if(\mu-\varepsilon_i)\ \text{if}\ i\in I^0(P)\cup I^+(P),\\ f(\mu+\varepsilon_i)\ \text{if}\ i\in I^-(P).\end{cases}$$         
       
          Let $\varphi(\mu)$ be a function satisfying
          $$\varphi(\mu+\varepsilon_i)=\begin{cases}(\mu_i+1)\varphi(\mu) \text{if}\ i\in I^0(P)\cup I^+(P),\\ -(\mu_i+1)\varphi(\mu)\ \text{if}\ i\in  I^-(P).\end{cases}$$
          Define a pairing $P\times P^F\to\mathbb C$ by setting $\langle e(\mu),f(\nu)\rangle=\varphi(\mu)\delta_{\mu,\nu}$. Then we have
          $$\langle \partial_i e(\mu),f(\nu)\rangle=-\langle  e(\mu),\partial_if(\nu)\rangle,\quad \langle x_i e(\mu),f(\nu)\rangle=\langle  e(\mu),x_if(\nu)\rangle.$$
          Hence $$\langle g(x)\partial_ie(\mu),f(\nu)\rangle=-\langle e(\mu), \partial_ig(x)f(\nu)\rangle.$$
          Using that $\partial_ig(x)=g(x)\partial_i+\partial_i(g(x))$ and choosing nonzero $\omega\in \bigwedge\nolimits^n {\mathbb C}^n$, we obtain
          $$g(x)\partial_i(f(\nu)\otimes\omega)=(\partial_ig(x)f(\nu))\otimes\omega.$$
This leads to a nondegenerate  ${\mathcal W}_n$-invariant pairing $P\times T(P^F, \bigwedge\nolimits^n {\mathbb C}^n)\to\mathbb C$.
          \end{proof}

  \begin{lemma} Let $V$ and $P$ be such that $T(P,V)$ has finite weight multiplicities, and let $V_*$ be the restricted dual of $V$. Then  $T(P^F,V_*\otimes\bigwedge\nolimits^n {\mathbb C}^n)$ and
    $T(P,V)$ are restricted dual to each other in the category of weight ${\mathcal W}_n$-modules.
  \end{lemma}
  \begin{proof} We define a pairing
    $$T(P^F,V_*\otimes\bigwedge\nolimits^n {\mathbb C}^n)\times T(P,V)\to\mathbb C$$
    by the formula
    $$\langle f\otimes v,g\otimes w\rangle=\langle f,g\rangle\langle v,w\rangle,\quad v\in V,w\in V_*, f\in P,g\in T(P^F,\bigwedge\nolimits^n {\mathbb C}^n). $$
    Then we have
    $$\langle x^\alpha\partial_j(f)\otimes v+\sum_{i}\partial_i(x^\alpha)f\otimes E_{ij}v,g\otimes w\rangle+  \langle f\otimes v,x^\alpha\partial_j(g)\otimes w+\sum_{i}\partial_i(x^\alpha)\otimes E_{ij}w\rangle=$$
    $$\langle x^\alpha\partial_j(f),g\rangle\langle v,w\rangle+\langle f, x^\alpha\partial_j(g)\rangle \langle v,w\rangle+$$
    $$ \sum_j\langle \partial_j(x^\alpha)f,g\rangle\langle E_{ij}v,w\rangle+ \langle f,\partial_j(x^\alpha)g\rangle\langle v,E_{ij}w\rangle=0,$$
    because of 
    $$\langle x^\alpha\partial_j(f),g\rangle+\langle f, x^\alpha\partial_j(f)\rangle =0,$$
    $$\langle \partial_j(x^\alpha)f,g\rangle= \langle f,\partial_j(x^\alpha)g\rangle$$
    and
    $$\langle E_{ij}v,w\rangle+\langle v,E_{ij}w\rangle=0.$$

    \end{proof}
\subsection{Statement of Main Result}
In this subsection we state and prove the main result in the paper. Some of the results used in the proof will be established in the next three sections. 

\begin{theorem}
Let $M$ be a simple weight ${\mathcal W}_n$-module with finite weight multiplicities. Then $M$ is the unique submodule of some tensor module $T(P,V)$ with finite weight multiplicities. More precisely, exactly one of the following holds:
\begin{itemize}
\item[(i)] $M$ is isomorphic to $T(P,V)$ for a simple weight $\mathcal D_n$-module $P$ and a simple weight $\mathfrak{gl}(n)$-module $V$  with finite weight multiplicities, such that $(\Delta^I_P\sqcup\Delta^-_P)\subset(\Delta^F_V\sqcup \Delta^-_V)$ and such that  $V$ is not isomorphic to a fundamental representation. 
\item[(ii)] $M$ is isomorphic to $dT(P, \bigwedge\nolimits^k {\mathbb C}^n)$ for some $k=0,1,...,n-1$, and a simple weight $\mathcal D_n$-module $P$.
\item[(iii)] $M \simeq \mathbb C$, which is the unique simple submodule of $T(P, \bigwedge\nolimits^0 {\mathbb C}^n)$.
\end{itemize} 

\end{theorem}
\begin{proof} Let $M$ be a simple weight ${\mathcal W}_n$-module with finite weight multiplicities.  By Theorem \ref{PS} and Proposition \ref{bounded},  $M$ is a quotient of the parabolically induced module
  $\operatorname{Ind}^{{\mathcal W}_n}_\p N$ where $N$ is a simple bounded $\g$-module over the  Levi subalgebra $\g$ of $\p$. Moreover, by Corollary \ref{conditions}, $N$ satisfies the additional conditions (1) and (2) of Section \ref{Levi}. 
  Theorem \ref{mainparabolic} provides a classification of such $N$. Finally, Lemma \ref{non-degeneratecase} and Lemma \ref{degeneratecase} ensure that $M$ is one of the modules listed in the statement.
\end{proof}

\section{Applications of the parabolic induction} \label{sec-par-ind}
Recall that $\Delta$ stands for the set of roots of ${\mathcal W}_n$. We use the setting of \S \ref{subsec-par-ind}.

In what follows we always assume that $M$ is a simple weight $\mathcal W_n$-module that has finite weight multiplicities. We will use that $M$ is the unique simple quotient of a parabolically induced module $U({\mathcal W}_n)\otimes_{U(\p)}M_0$, as stated in  Theorem \ref{PS}. Let $\p$ be the parabolic subalgebra associated
with $\gamma=\sum_{i=1}^n a_i\varepsilon_i$. We assume without loss of
generality that
$$a_1\geq\dots\geq a_p>0=a_{p+1}=\dots=a_{p+m}>a_{p+m+1}\geq\dots\geq a_n.$$  Henceforth we fix $\p$ and denote by $\g$ the Levi subalgebra of $\p$. Then $\g\simeq {\mathcal W}_m\ltimes(\k\otimes \mathcal O_m)$ where $\k$ is a Levi subalgebra in
$\mathfrak{gl}(p)\oplus\mathfrak{gl}(n-m-p)$.
Under this assumptions we have the following 

\begin{proposition}\label{bounded} The simple $\g$-module $M_0$ is bounded.
  \end{proposition}
  \begin{proof} First we prove three preliminary results. 
    \begin{lemma} Let $\alpha=-\varepsilon_i$ or $\alpha\in\Delta(\k)$. Then
      \begin{enumerate}
      \item $\dim\g_\alpha=1$ and any nonzero $X_{\alpha} \in \g_\alpha$ can be included in the $\mathfrak{sl}_2$ triple;
      \item Either $\g_\alpha$ acts locally nilpotently on $M_0$ or $\g_\alpha:M_0\to M_0$ is injective.
        \end{enumerate}
    \end{lemma}
    \begin{proof} The first assertion is obvious. The second follows from the fact that $\ad\g_\alpha$ is locally nilpotent in $\g$.
    \end{proof}
    \begin{lemma} Let $\alpha=-\varepsilon_i$ for $p<i\leq p+m$ or $\alpha\in\Delta(\k)$. Then $\g_\alpha$ acts injectively on $M_0$.
    \end{lemma}
    \begin{proof} Suppose that $\g_\alpha$ is locally nilpotent on $M_0$. Let $h$ be the Cartan element in the $\mathfrak{sl}_2$-triple containing $X_{\alpha} \in \g_\alpha \setminus \{ 0 \}$. In particular, $\alpha(h)=2$. Let $\mu\in\supp M$. Then $\mu+\mathbb Z\alpha\subset\supp M$.
      Furthermore, for any $n>0$ there exist $k\geq n$ and $v\in M^{\mu+k\alpha}$ such that $\g_\alpha v=0$. Let $M_k$ denote the $\mathfrak{sl}(2)$-submodule of $M$ generated by $v$. For all sufficiently large $k$
      we have $\mu\in \supp M_k$. Therefore $\dim M^\mu=\infty$. A contradicton.
      \end{proof}
      \begin{corollary}\label{c:conditions}  Let $\alpha=-\varepsilon_i$ or $\alpha\in\Delta(\k)$. For any $\lambda\in\supp M$ and $X\in\g_\alpha\setminus 0$ the map $X:M^\lambda\to M^{\lambda+\alpha}$ is an isomorphism.
      \end{corollary}
      \begin{proof} From the previous lemma we know that $X:M^\lambda\to M^{\lambda+\alpha}$ is injective. Applying the same lemma to $M_*$ we obtain $X:M_*^{-\lambda-\alpha}\to M_*^{-\lambda}$ is injective. Hence
        $X:M^\lambda\to M^{\lambda+\alpha}$ is surjective.
      \end{proof}
   We are now ready to complete the proof of Proposition \ref{bounded}.    Corollary \ref{c:conditions} implies $\dim M^\mu=\dim M^{\mu+\gamma}$ for any $\gamma\in \mathbb Z\Delta({\mathcal W}_m)+\mathbb Z\Delta(\k)$ and $\mu\in\supp M$. The statement
      follows.
    \end{proof}

    \section{Bounded simple $\g$-modules}\label{Levi}
    \subsection{Generalization of tensor modules for the Levi subalgebra $\g$ of $\p$} We retain the notation of the previous section. In this section we assume that  $m>0$. Recall that 
       $\g={\mathcal W}_m\ltimes(\k\otimes \mathcal O_m)$. Without loss of generality we may assume $\mathcal O_m=\mathbb C[x_1,\dots,x_m]$.
    In this section we will classify simple bounded $\g$-modules $N$ satisfying the additional properies:
    \begin{enumerate}
    \item $\supp N=\lambda+\mathbb Z\Delta(\g)$ for any $\lambda\in\supp N$.
      \item All weight spaces of $N$ have the same dimension $d$.
      \end{enumerate}
It follows from the proof of Proposition \ref{bounded} that $\g_\alpha$ acts injectively on $N$ if $\alpha=-\varepsilon_i$ or $\alpha\in\Delta(\k)$.
  
First, we generalize the notion of a $({\mathcal W}_m,\mathcal O_m)$-module to that of a $(\g,\mathcal O_m)$-module.
\begin{definition}\label{def1} A $\g$-module $N$ is a $(\g,\mathcal O_m)$-module if
  $N$ is a $\mathcal O_m$-module satisfying
  \begin{equation}\label{cond1}
    X(fv)=fX(v)+X(f)v\  \forall v\in N, f\in\mathcal O_m, X\in {\mathcal W}_m,
  \end{equation}
 \begin{equation}\label{cond2}  
   (h\otimes Y)(fv)=(hf)Yv\  \forall v\in N, \; f,h\in\mathcal O_m, Y\in \k.
   \end{equation}
  \end{definition}

  From now on we assume that all  $({\mathcal W}_m,\mathcal O_m)$-modules and all $(\g,\mathcal O_m)$-modules are weight modules.
  
  \begin{remark}\label{associative} 
    Consider the algebra $\mathcal A(m)$ generated by ${\mathcal W}_m\otimes 1$ and $1\otimes \mathcal O_m$ with relations
    $$(x\otimes 1)(y\otimes 1)-(y\otimes 1)(x\otimes 1)=[x,y]\otimes 1,$$
    $$(1\otimes f)(1\otimes g)=1\otimes fg,$$
    $$(x\otimes 1)(1\otimes f)-(1\otimes f)(x\otimes 1)=1\otimes x(f)$$
for $x,y\in {\mathcal W}_m$ and $f,g\in\mathcal O_m$.    
Any $({\mathcal W}_m,\mathcal O_m)$ is an $\mathcal A(m)$-module, and conversly, any $\mathcal A(m)$-module is a $({\mathcal W}_m,\mathcal O_m)$-module.
Furthermore, $\mathcal A(m)$ is isomorphic to $U({\mathcal W}_m)\otimes\mathcal O_m$ as a vector space by the correspondence $(X\otimes 1)(1\otimes f)\mapsto X\otimes f$ for all $X\in U({\mathcal W}_m)$ and $f\in\mathcal O_m$.    
    Let $\mathcal B:=\mathcal A(m)\otimes U(\k)$. Then any  $(\g,\mathcal O_m)$-module is a $\mathcal B$-module. 
    \end{remark}

\begin{example}     Let $S$ be a $\k$-module.  We define a $(\g,\mathcal O_m)$-module structure on the vector space
  $\mathcal O_m\otimes S$ by setting
  $$f(h\otimes s)=fh\otimes s,\quad (f\otimes Y)(h\otimes s)=fh\otimes Ys,\quad X(h\otimes s)=X(h)\otimes s$$
  for all $f,h\in\mathcal O_m$, $Y\in \k$, $X\in {\mathcal W}_m$ and $s\in S$.
  One can easily verify that $\tilde S:=\mathcal O_m\otimes S$ is a $(\g,\mathcal O_m)$-module.
  Moreover, if $R$ is a $({\mathcal W}_m,\mathcal O_m)$-module then $\mathcal F(R,S):=R\otimes _{\mathcal{O}_m}\tilde S$ is a $(\g,\mathcal O_m)$-module. 
\end{example}
\begin{remark}\label{chinese} A simple weight  $({\mathcal W}_m,\mathcal O_m)$-module $R$ with finite weight multiplicities is a tensor module $T(P,V)$ for some simple weight $\mathcal D_m$-module $P$ and some simple weight $\mathfrak{gl}(m)$-module $V$, see Theorem 3.7 in \cite{XL}. 
\end{remark}

  \begin{lemma}\label{irreducibility} If $R$ is a simple $({\mathcal W}_m,\mathcal O_m)$-module and $S$ is a simple weight $\k$-module then $\mathcal F(R,S)$ is a simple $(\g,\mathcal O_m)$-module, in the sense that it does not
    contain proper nontrivial   $(\g,\mathcal O_m)$-submodules.
\end{lemma}
\begin{proof} Observe that $\mathcal F(R,S)$ is isomorphic to $R\otimes S$ as a $\mathcal B$-module. Hence it is a simple $\mathcal B$-module. This implies the statement.
\end{proof}
\begin{lemma}\label{conditions} If $N=\mathcal F(R,S)$ satisfies conditions (1) and (2), then $S$ is a simple cuspidal $\k$-module, and $R=T(P,V)$ for some simple cuspidal $\mathcal D_m$-module $P$  a simple finite-dimensional  $\mathfrak{gl}(m)$-module $V$. For any $\lambda\in\supp N$ we have that $\dim N^\lambda=(\dim V)d(S)$, where $d(S)$ is the degree of the cuspidal module $S$. 
\end{lemma}
\begin{proof} The lemma follows from the isomorphism of $\h$-modules $\mathcal F(R,S)\simeq R\otimes S$. 
  \end{proof}
  \begin{lemma} \label{simplicity} Let $S$ be a simple nontrivial weight $\k$-module, $P$ be a simple weight $\mathcal D_m$-module, and $V$ be a simple finite-dimensional $\mathfrak{gl}(m)$-module.
    Then $\mathcal F(T(P,V),S)$ is a simple $\g$-module.
  \end{lemma}
  \begin{proof} Choose a regular $u\in\left( \k\cap\h\right)^*$ that acts nontrivially on $S$, and denote by $\mathcal F(T(P,V),S)^a$ the eigenspace of $u$ with eigenvalue $a$.
    Let $M$ be a proper nonzero submodule of $\mathcal F(T(P,V),S)$. Then $M^a=M\cap\mathcal F(T(P,V),S)^a$ is $\mathcal O_m$-invariant for any $a\neq 0$. Using the action of the root elements of $\k$, we obtain that $M^a$  is $\mathcal O_m$-invariant for $a=0$ as well. 
   Hence $M$ is a $(\g,\mathcal O_m)$-submodule of $\mathcal F(T(P,V),S)$ and we reach a contradiction.
    \end{proof}

    \begin{lemma}\label{geometric} Let $N$ be a simple $(\g,\mathcal O_m)$-module satisfying (1) and (2). Then $N$ is isomorphic to $\mathcal F(R,S)$ for some $({\mathcal W}_m,\mathcal O_m)$-module
      $R$ and some simple  cuspidal $\k$-module $S$.
    \end{lemma}
    \begin{proof} Recall the definition of $\mathcal B$ from Remark \ref{associative}. Consider  $N$ as a $\mathcal B$-module. By definition, for any vector $v\in N$ we have $\mathcal B v=U(\g)v$ (this follows from the relation $(f\otimes Y)v=f(Yv)$). Hence, $N$ is a simple $\mathcal B$-module.
      For a simple $\k$-module $S'$, the subspace $\Hom_\k(S',N)\otimes S'$ of $N$ is $\mathcal B$-stable. Hence, there is a unique up to isomorphism $S'$, such that $\Hom_\k(S',N)\neq 0$. The existence of such $S'$ follows from
      condition on the support of $N$. Let $S$ be such module. We have that  $R=\Hom_\k(S,N)$ is a simple $\mathcal A(m)$-module. Therefore, $N\simeq R\otimes S$ as a $\mathcal B$-module. The condition (\ref{cond2}) ensures that $N\simeq \mathcal F(R,S)$.      
    \end{proof}
    Recall that $\tilde V$ stands for $(\mathcal W_m,\mathcal O_m)$-module $\mathcal O_m \otimes V$.
    \begin{lemma}\label{degenerate} Let $N$ be a simple  weight $\g$-module, such that $\partial_i$ acts locally  nilpotently for all $i=1,\dots, m$. Then $N$ is isomorphic to a simple submodule of $\mathcal F(\tilde V,S)$ for
      some simple
      $\k$-module $S$ and a  simple $\mathfrak{gl}(m)$-module $V$.
    \end{lemma}
    \begin{proof} Let $N_0$ be the space of invariants of $\partial_1,\dots,\partial_m$. If $\q$ is the parabolic subalgebra associated to  $\gamma = -(\varepsilon_1+\dots+\varepsilon_m)$, then $N$ is the unique simple
      quotient of $U(\g)\otimes_{U(\q)}N_0$. Thus, $N_0$ is a simple $\mathfrak{gl}(m)\oplus\k$-module, so $N_0=V\otimes S$ for some simple modules $V$ and $S$. Then we have a natural  homomorphism $\varphi:N_0\to {\mathcal F}(\tilde V,S)$ of $\mathfrak{gl}(m)\oplus\k$-modules, hence, also of $\q$-modules. The homomorphism $\varphi$  induces a homomorphism $\Phi:U(\g)\otimes_{U(\q)}N_0\to\mathcal F(\tilde V,S)$ of $\g$-modules. The image of $\Phi$ is isomorphic to $N$.
          \end{proof}
    \begin{lemma}\label{genau} Let $N$ be  a simple $\g$-module satisfying (1) and (2). Assume that one can define an  $\mathcal O_m$-module structure on $N$ in such a way that it satisfies (\ref{cond1}) and
      $$f(g\otimes Y)v=(g\otimes Y)fv,\ \forall v\in N, f,g\in\mathcal O_m, Y\in\k.$$
      Then $N$ is a $(\g,\mathcal O_m)$-module.
    \end{lemma}
    \begin{proof} We need to verify (\ref{cond2}). First, we claim that verifying (\ref{cond2}) is equivalent to checking
\begin{equation}\label{eq1}
  x_1(1\otimes Y)=(x_1\otimes Y),\ \forall Y\in\k.
  \end{equation}
      Indeed, let $f\in \mathcal O_m$ and $X=f\partial_1$. Then
      $$[X,x_1](1\otimes Y)=f(1\otimes Y)=[X,x_1\otimes Y]=f\otimes Y.$$
    For $f,g\in\mathcal O_m$ we have
      $$f(g\otimes Y)=fg(1\otimes Y).$$
      Next, observe that (1) and (2) ensure that $\partial_i$ and $x_i$ act injectively on $N$. Therefore we can localize $N$ with respect $x_i$ and define a $({\mathcal W}_m,\tilde{\mathcal O}_m)$-module structure on $N$,
      where $\tilde{\mathcal O}_m=\mathbb C[x_1^{\pm 1},x_2^{\pm 1},\dots,x_m^{\pm 1}]$. Consider the twisted localization $D_{\langle x_1,\dots ,x_n\rangle }^{\mathbf c} N$ of $N$ with $\mathcal U=\mathcal B$, and some $\mathbf c\in\mathbb C^m$.
      We can choose $\mathbf c$ so that $D_{\langle x_1,\dots ,x_n\rangle }^{\mathbf c} N$   has a nonzero weight vector  annihilated by all $\partial _i$.
      Among all such vectors choose one of weight $\mu$ with maximal  possible $|\mu|_r:=\operatorname{Re}\sum_{i=1}^m\mu_i$. Let $N'$
      be a $\g$-submodule of  $D_{\langle x_1,\dots ,x_n\rangle }^{\mathbf c} N$ generated by $u$. Note that for any $\nu\in \supp N'$ we have $|\mu|_r\geq|\nu|_r$. Let $v$ have weight $\nu$ with $|\nu|_r=|\mu|_r$. Then $\partial_iv=0$ and 
      $$\partial_i( x_1(1\otimes Y)-x_1\otimes Y)v=0,$$
      hence $u=( x_1(1\otimes Y)-x_1\otimes Y)v$ is annihilated by all $\partial_i$. On the other hand, the weight $\eta$ of $u$ satisfies $|\eta|_r=|\mu|_r+1$ hence $u=0$.
      Let $w\in N'$ be a weight vector of weight $\lambda$ with minimal $|\lambda|$ such that for some $Y\in\k$
      $$(x_1(1\otimes Y)-x_1\otimes Y)w\neq 0.$$
      We have
      $$\partial_i(x_1(1\otimes Y)-x_1\otimes Y)w=(x_1(1\otimes Y)-x_1\otimes Y)\partial_iw=0,$$
     which leads to a contradiction.
      Next we note that $D_{\langle x_1,\dots ,x_n\rangle }^{\mathbf c} N=\tilde{\mathcal O}_m \cdot N'$. Since $x_1(1\otimes Y)-x_1\otimes Y$ commutes with $\tilde{\mathcal O}_m$ we have $(x_1(1\otimes Y)-x_1\otimes Y)N^{(\mathbf c)}=0$. Then
      $(x_1(1\otimes Y)-x_1\otimes Y)N=0$. This completes the proof.
    \end{proof}
    \begin{lemma}\label{center} Assume that $N$ is a simple $\g$-module satisfying (1) and (2). Let $z$ be a central element of $\k$ which does not act trivially on $N$. Then $N$ is a $(\g,\mathcal O_m)$-module and hence is isomorphic
      to a module $\mathcal F(R,S)$.
    \end{lemma}
    \begin{proof} Without loss of generality we may assume that $z$ acts as identity on $N$. Define an $\mathcal O_m$-module structure on $N$ by setting $x_i v:=(x_i\otimes z)v$. Then $N$ satisfies the assumptions of Lemma \ref{genau}.
      The statement follows.
      \end{proof}
       \begin{lemma}\label{abelian} Let $\k$ be abelian and $N$ be a simple $\g$-module satisfying (1) and (2). 
If $\k N=0$, then $(\mathcal O_m\otimes\k) N=0$.
      \end{lemma}
      \begin{proof} We will show that $(f\otimes h) N=0$ for any $f\in\mathcal O_m,\ h\in\k$. Note that
        $$[\partial_1,x_1\otimes h]N=hN=0.$$ Therefore we have the following identites on $N$:
        $$[\partial_1(x_1\otimes h),\partial_1^2(x_1^2\otimes h)]=2\partial_1^2(x_1\otimes h)^2=2(\partial_1(x_1\otimes h))^2,$$
        $$[(\partial_1(x_1\otimes h))^k,\partial_1^2(x_1^2\otimes h)]=2k(\partial_1(x_1\otimes h))^{k+1}.$$
        This implis that, on each weight space $N^\lambda$ of $N$, $\operatorname{tr}_{N\lambda}(\partial_1(x_1\otimes h))^k=0$ for all $k>2$. Since $N$ is bounded this implies nilpotency of $\partial_1(x_1\otimes h)$ on $N$. Since $\partial_1$
        is invertible on $N$ we obtain that $x_1\otimes h$ is nilpotent on $N$. Let $p$ be the nilpotency degree of $x_1\otimes h$. There exists $v\in N$ such that
        $w:=(x_1\otimes h)^{p-1}v\neq 0$. Then for  $f\in \mathcal O_m$ we have
        $$0=f\partial_1(x_1\otimes h)^pv=p(f\otimes h)(x_1\otimes h)^{p-1}v.$$
        In other words, $w$ is annihilated by $\mathcal O_m\otimes h$. The subspace $N'$ of all vectors annihilated by $\mathcal O_m\otimes h$ is $\g$-invariant, but we just proved that $N'\neq 0$.
        By the  irreducibility of $N$, we have $N=N'$. Thus $(\mathcal O_m\otimes h)N=0$. \end{proof}

      \begin{proposition}\label{reduction} Let $N$ be a $\g$-module satisfying (1) and (2). Then for any Borel subalgebra $\b\subset\k$ there exists a simple bounded $\g$-module $\bar N$ satisfying the following two conditions:
        \begin{enumerate}
        \item[(i)] There exists a weight $\lambda\in\h^*$ such that $\supp \bar N\subset\lambda+\sum_{i=1}^m\mathbb Z\varepsilon_i-{{\mathbb Z}_{\geq 0}}\Delta(\b)$ and $\lambda(\h\cap\k)\neq 0$.
          \item[(ii)] The module $N$ is obtained from $\bar N$ by a twisted localization with respect to some set of commuting roots $\Gamma\subset -\Delta(\b)$.
          \end{enumerate}
        \end{proposition}

        \begin{proof} Let $\hat\k=\k+\h$, $\mathfrak n=[\b,\b]$. Then $N$ is a bounded weight $\hat\k$-module, hence, any  cyclic $\hat\k$-submodule of $N$ has finite length (see Lemma 3.3 in \cite{M}). Let $N_0$ be a simple $\hat\k$-submodule of $N$.
          Note that $N_0$ is a cuspidal $\hat\k$-module.
  By  Proposition 4.8 in \cite{M}, there exist $\mu\in(\h\cap\k)^*$ and $\Gamma\subset -\Delta(\b)$ such that $N_0  \simeq D^{\mu}_{\Gamma}M_0$ for some simple bounded $\mathfrak b$-highest weight $\hat\k$-module $M_0$.
  Since $D^{\mu}_{\Gamma}$
  is well defined for $\g$-modules and commutes with the restriction functor $\operatorname{Res}^\g_{\hat\k}$, $M:=D^{-\mu}_{\Gamma}N$ contains an $\mathfrak n$-primitive weight vector $v \in M_0$, while
  $\partial_i$  act injectively on $M$ for all $i=1,\dots,m$. 
  Since $U(\hat\k)v$ is bounded it has finite $\hat\k$-length and hence there is $\lambda'\in \supp U(\hat\k)v$ such that  $\lambda'+\alpha\notin \supp U(\hat\k)v$ for all $\alpha\in\Delta(\b)$. The injectivity of  the action of $\partial_i$ implies that
  $$(\lambda'+\alpha+\sum_{i=1}^m {\mathbb Z}_{\geq 0} \varepsilon_i)\cap\supp U(\g)v=\emptyset.$$
  If $w$ is a nonzero vector of weight $\lambda'$ then
  $$\supp U(\g)w\subset\lambda'+\sum_{i=1}^m\mathbb Z\varepsilon_i-{{\mathbb Z}_{\geq 0}}\Delta(\b).$$
  This implies  $\dim (U(\mathcal O_m\otimes\mathfrak n))u<\infty$ for any $u\in U(\g)w$. By Lemma \ref{bnd-wn-finite-length}, the boundedness of $U(\g)w$ implies that $U(\g)w$ has finite length.
  Let $\bar N$ be a simple submodule of $U(\g)w$.
  Then there is a nonzero weight vector $u\in\bar N$ annihilated by $\mathcal O_m\otimes\mathfrak n$. Then $\bar N$ satisfies (i) with $\lambda$ being the weight of $u$,
  while (ii) follows from the simplicity of $N$.

  It remains to show that $\lambda(\h\cap\k)\neq 0$. For the sake of contradiction, assume that the opposite holds. Take a simple root $\alpha\in\Delta(\b)$. Then a simple computation shows that $\g_{-\alpha}u$ is annihilated by $\mathcal O_m\otimes\mathfrak n$.
  The simplicity of $\bar N$ hence implies that $\g_{-\alpha}u=0$ for all simple roots $\alpha$ and thus $M$ contains a trivial $\k$-submodule. But the roots of $\Gamma$ act injectively on $N$ and hence on $M$. This leads to  a contradiction.
   \end{proof}
   For a weight $\mu \in (\h \cap \k)^*$ and a Borel subalgebra $\b$ of $\k$, by $L_{\b} (\mu)$ (or simply by  $L (\mu)$ ) we denote the simple $\b$-highest weight $\k$-module of highest weight $\mu$.
   \begin{lemma}\label{tens_reduction} The module $\bar N$ constructed in Proposition \ref{reduction} is isomorphic to $\mathcal F(T(P,V),L(\bar\lambda))$ for a cuspidal simple $\mathcal D_m$-module $P$, a simple
     finite-dimensional $\mathfrak{gl}(m)$-module $V$, and a simple highest weight $\k$-module $L(\bar\lambda)$, where $\bar\lambda$ is the restriction of $\lambda$ to $\h\cap\k$.
   \end{lemma}
   \begin{proof} Consider $\gamma\in \left( \k\cap\h \right)^*$ which determines the Borel subalgebra $\b$. Then $\gamma$ determines also a parabolic subalgebra $\q$ in $\g$. The $\q$-top of the module $\bar N$ is a simple $({\mathcal W}_m\oplus\h)$-module. Since
     $\bar\lambda\neq 0$, this module is isomorphic to $T(P,V)\otimes\mathbb C_{\bar\lambda}$ by Lemma \ref{center}. A simple computation shows that it is also isomorphic to the top of $\mathcal F(T(P,V),L(\bar\lambda))$. Hence
     the statement follows from Proposition \ref{parabolicind}(c).
   \end{proof}
   \begin{corollary}\label{non-abelian} If $\k$ is not abelian then $N$ is isomorphic to $\mathcal F(T(P,V),S)$ for some cuspidal simple $\mathcal D_m$-module $P$,  a simple
     finite-dimensional $\mathfrak{gl}(m)$-module $V$, and a simple cuspidal $\k$-module $S$. 
   \end{corollary}
   \begin{proof} The result follows immediately from Proposition \ref{reduction}, Lemma \ref{tens_reduction}, and the isomorphism of $\g$-modules
     $$D^{-\mu}_\Gamma\mathcal F(T(P,V),L(\bar\lambda))\simeq \mathcal F(T(P,V),D^{-\mu}_\Gamma L(\bar\lambda)).$$
     \end{proof}

        \begin{theorem}\label{mainparabolic} Let $N$ be a simple bounded $\g$-module satisfying (1) and (2). Then we have one of the following two mutually exclusive statements.

          (a) $\mathcal O_m\otimes\k$ acts trivially on $N$ and $N$ is a unique simple submodule of $T(P,V)$ for some simple cuspidal ${\mathcal D}_m$-module $P$ and a simple finite-dimensional $\mathfrak{gl}(m)$-module $V$.
          In this case $\k$ must be abelian.

          (b) $N$ is isomorphic to $\mathcal F(T(P,V),S)$ for some cuspidal simple $\mathcal D_m$-module $P$,  a simple
     finite-dimensional $\mathfrak{gl}(m)$-module $V$ and a simple nontrivial cuspidal $\k$-module $S$. 
\end{theorem}
\begin{proof} If $\k$ is not abelian the statement follows from  Corollary \ref{non-abelian}. If $\k$ is abelian and  $\k$ acts nontrivially on $N$, the statement follows from Lemma \ref{center}.
          If $\k$ acts trivially on $N$, then by Lemma \ref{abelian}, $(\mathcal O_m\otimes\k) N=0$. Then $N$ is a simple bounded ${\mathcal W}_m$-module and the statement is a consequence of Theorem 1.1 in \cite{XL}.
    \end{proof}

    \section{Back to tensor modules via parabolic induction}
\subsection{The case of infinite-dimensional $\g$}
        We retain the notation of Section 4 and  assume again that $M$ is a simple weight ${\mathcal W}_n$-module that is also the  unique simple quotient of the parabolically induced module $U({\mathcal W}_n)\otimes_{U(\p)}N$, where  $N$ is a simple bounded $\g$-module satisfying (1) and (2). We we will use the properties of $N$ listed  in
        Theorem \ref{mainparabolic}.

     Recall that $p$ and $m$ are fixed and defined in \S \ref{sec-par-ind}.   Let $\p'=\p\cap\mathfrak{gl}(n)$. The Levi subalgebra of $\p'$ is isomorphic to $\k\oplus\mathfrak{gl}(m)$.
        Consider a $\g$-module $\mathcal F(T(P,V),S)$ where $V$ is a finite-dimensional $\mathfrak{gl}(m)$-module, $P$ is a simple cuspidal $\mathcal D_m$-module and $S$ be a simple cuspidal $\k$-module.
        Note that $S$ might be a trivial $\k$-module in the case when $\k$ is abelian.
        Let $U$ be the one-dimensional $\k$-module of weight $\sum_{i=1}^p\varepsilon_i$ 
        and $S^U=S\otimes U$. Finally, let
        $\hat S$ be the unique simple quotient of $U(\mathfrak{gl}(n))\otimes_{U(\p')}(S^U\otimes V)$. Using the isomorphism $$\mathcal D_n\simeq \mathcal D_p\otimes \mathcal D_m \otimes \mathcal D_{n-p-m},$$
        define a $\mathcal D_n$-module $\tilde P$ by
        $$\tilde P=\mathbb C[x_1,\dots,x_p]^F\otimes P\otimes\mathbb C[x_{p+m+1},\dots,x_n],$$
        (recall that $X^F$ is the full Fourier transform of $X$).

        \begin{lemma}\label{backtotensor} The $\p$-top of $T(\tilde P, \hat S)$ is isomorphic to $\mathcal F(T(P,V),S)$.
          \end{lemma}
          
          \begin{proof}
          The statement follows by comparing the supports of the two modules. Let $\p = \g \oplus \n $. 
          \begin{eqnarray*}
          \supp \tilde P = \sum_{i=1}^p {\mathbb Z}_{<0} \varepsilon_i + \supp P +  \sum_{i=p+m+1}^n {\mathbb Z}_{\geq 0} \varepsilon_i, \\ \supp \hat S \subset  \supp S + \sum_{i=1}^p \varepsilon_i + \supp V -\supp U(\n'),
          \end{eqnarray*}
           where $\n'=\n\cap\mathfrak{gl}(n)$.
           Then we have that         
        $$     \supp \mathcal F(T(P,V),S)\subset \supp T(\tilde P,\hat S)\subset \supp \mathcal F(T(P,V),S)+\supp U(\n^-),$$
          where $\n^-$ is the nilradical of the opposite parabolic.  Moreover, the multiplicity of any $\mu\in  \supp \mathcal F(T(P,V),S)$ is the same as its multiplicity in $T(\tilde P, \hat S)$.
        \end{proof}
        \begin{lemma}\label{non-degeneratecase} Let $N= \mathcal F(T(P,V),S)$ be a simple $\g$-module. Then the unique simple quotient $M$ of $U({\mathcal W}_n)\otimes_{U(\p)}N$ is isomorphic to unique simple submodule of $T(\tilde P, \hat S)$. 
     \end{lemma}
     \begin{proof}  The isomorphism  of $\p$-modules $N\to T(\tilde P,\hat V)^{\rm top}$ induces a nonzero homomorphism of ${\mathcal W}_m$-modules $U({\mathcal W}_n)\otimes_{U(\p)}N\to T(\tilde P,\hat V)$. The image of this homomorphism is simple since
       $T(\tilde P,\hat V)$ has a unique simple submodule. Thus, this submodule is isomorphic to $M$.
     \end{proof}
     Now assume that  $\mathcal F(T(P,V),S)$ is not simple. This is only possible if $\k$ is abelian, $S$ is trivial, and $V=\bigwedge\nolimits^{k} {\mathbb C}^m$. 
     
    \begin{lemma}\label{degeneratecase} Assume that $N$ is the simple submodule $T(P, \bigwedge\nolimits^{k} {\mathbb C}^m)$ for some
     $k=0,\dots,m-1$. Then the unique simple quotient $M$ of $U({\mathcal W}_n)\otimes_{U(\p)}N$ is isomorphic to the unique simple submodule of $T(\tilde P, \bigwedge\nolimits^{p+k} {\mathbb C}^n)$. 
     \end{lemma}    
     \begin{proof} We consider the monomorphism of $\p$-modules $N\to T(\tilde P, \bigwedge\nolimits^{p+k} {\mathbb C}^n)^{\rm top}$, and the induced map $$U({\mathcal W}_n)\otimes_{U(\p)}N\to T(\tilde P, \bigwedge\nolimits^{p+k} {\mathbb C}^n).$$
       To complete the proof, we use the same reasoning as the one in the proof of the previous lemma.
       \end{proof}
   
       \subsection{The case of finite-dimensional $\g$} In this case we have $m=0$ and $\g$ is a Lie  subalgebra of $\mathfrak{gl}(n)$.  Using arguments similar to the ones used in the previous subsection, one can show that
       the unique simple quotient of $U({\mathcal W}_n)\otimes_{U(\p)}S$ is isomorphic to the unique simple submodule of $T(\tilde P,\hat S)$.
       
       \subsection {The case $\g={\mathcal W}_m$} This case follows from Theorem \ref{thm-class-bnd}.

    \end{document}